\documentclass[oneside,reqno,english]{amsart}
\usepackage[T1]{fontenc}
\usepackage[utf8]{inputenc}
\usepackage{xcolor}
\usepackage{babel}
\usepackage{prettyref}
\usepackage{amstext}
\usepackage{amsthm}
\usepackage{amssymb}
\usepackage[all]{xy}
\usepackage[pdfusetitle,
 bookmarks=true,bookmarksnumbered=false,bookmarksopen=false,
 breaklinks=false,pdfborder={0 0 0},pdfborderstyle={},backref=false,colorlinks=false]
 {hyperref}
\hypersetup{
 colorlinks=true,citecolor=blue,linkcolor=blue,linktocpage=true}

\makeatletter
\numberwithin{equation}{section}
\numberwithin{figure}{section}

\usepackage{prettyref}

\newrefformat{cor}{Corollary~\ref{#1}}
\newrefformat{subsec}{Section~\ref{#1}}
\newrefformat{lem}{Lemma~\ref{#1}}
\newrefformat{thm}{Theorem~\ref{#1}}
\newrefformat{sec}{Section~\ref{#1}}
\newrefformat{chap}{Chapter~\ref{#1}}
\newrefformat{prop}{Proposition~\ref{#1}}
\newrefformat{exa}{Example~\ref{#1}}
\newrefformat{tab}{Table~\ref{#1}}
\newrefformat{rem}{Remark~\ref{#1}}
\newrefformat{def}{Definition~\ref{#1}}
\newrefformat{fig}{Figure~\ref{#1}}
\newrefformat{claim}{Claim~\ref{#1}}
\newrefformat{assu}{Assumption~\ref{#1}}
\newrefformat{prob}{Problem~\ref{#1}}

\makeatother

\theoremstyle{plain}
\newtheorem{thm}{\protect\theoremname}[section]
\theoremstyle{definition}
\newtheorem{problem}[thm]{\protect\problemname}
\theoremstyle{plain}
\newtheorem{lem}[thm]{\protect\lemmaname}
\newtheorem{cor}[thm]{\protect\corollaryname}
\theoremstyle{remark}
\newtheorem*{rem*}{\protect\remarkname}
\newtheorem{rem}[thm]{\protect\remarkname}
\theoremstyle{plain}
\newtheorem{prop}[thm]{\protect\propositionname}
\providecommand{\corollaryname}{Corollary}
\providecommand{\lemmaname}{Lemma}
\providecommand{\problemname}{Problem}
\providecommand{\propositionname}{Proposition}
\providecommand{\remarkname}{Remark}
\providecommand{\theoremname}{Theorem}

\begin{document}
\title{Normalized Orbit Frames and Essential Spectrum}
\begin{abstract}
We characterize the essential spectra of normal operators admitting
normalized orbit frames. These are exactly closed subsets of the unit
circle with positive Haar measure.
\end{abstract}

\author{James Tian}
\address{Mathematical Reviews, 535 W. William St, Suite 210, Ann Arbor, MI
48103, USA}
\email{james.ftian@gmail.com}

\maketitle
\tableofcontents{}

\section{Introduction}\label{sec:1}

Dynamical sampling asks whether measurements taken along the orbit
of an operator can replace measurements made at different spatial
locations. For a bounded operator $T$ on a Hilbert space, this leads
to families of the form $\left(T^{n}g\right)_{n\geq0}$ and, more
generally, to finitely many such orbits. The modern formulation was
introduced in \cite{MR3027915} and developed systematically in \cite{MR3613395}.
For normal and self-adjoint operators the spectral theorem gives a
more concrete form of the problem, and a substantial part of the earlier
work concerns the restrictions imposed by the spectrum.

For self-adjoint operators, a negative result for normalized iterates
was proved in the pure point dynamical sampling setting of \cite{MR3613395},
and the general self-adjoint case was settled in \cite{MR3579135}.
The corresponding existence question for general normal operators
remained open.

Krishtal and Pfander \cite{arXiv2606.20848} recently showed that
the answer is different for normal operators. They constructed a bounded
normal operator $T$ and a vector $g$ for which $\left(T^{n}g/\left\Vert T^{n}g\right\Vert \right)_{n\geq0}$
is a frame. Their example is diagonal, with the eigenvalues arranged
in rapidly growing finite blocks on circles whose radii tend to one.

Since a normalized orbit frame can occur only when $g$ is cyclic
for $T$, we consider arbitrary cyclic pairs $\left(T,g\right)$,
where $T$ is a bounded normal operator. Set 
\[
x_{n}:=\frac{T^{n}g}{\left\Vert T^{n}g\right\Vert },\qquad n\geq0.
\]
WLOG, we take $\left\Vert T\right\Vert =1$.

Let $\mu$ be the scalar spectral measure associated with $T$ and
$g$, supported on $\overline{\mathbb{D}}$. The spectral theorem
gives a unitary $U:H\longrightarrow L^{2}\left(\mu\right)$ such that
$Ug=1$ and $UTU^{-1}=M_{z}=$ multiplication by $z$. Set 
\begin{align*}
m_{n} & :=\int_{\overline{\mathbb{D}}}\left|z\right|^{2n}d\mu(z)=\left\Vert T^{n}g\right\Vert ^{2},\\
u_{n}\left(z\right) & :=\frac{z^{n}}{\sqrt{m_{n}}}=Ux_{n}.
\end{align*}
Then $\left(x_{n}\right)_{n\geq0}$ is a frame for $H$ if and only
if $\left(u_{n}\right)_{n\geq0}$ is a frame for $L^{2}\left(\mu\right)$.

The ordinary one-vector orbit problem has a different function theoretic
form. In the unnormalized setting considered in \cite{MR3613395},
the frame problem is expressed in the Hardy space and becomes a Carleson
interpolation problem. Here the moments of $\mu$ appear in the normalization,
so the analytic space depends on the spectral measure itself.

A natural problem is to characterize the finite positive measures
$\mu$ for which the normalized monomials $\left(u_{n}\right)_{n\geq0}$
form a frame. We do not give a solution of this measure problem. Instead,
the measure model leads to a spectral question about normal operators.
\begin{problem}
\label{prob:1-1}Which closed sets $E\subset\mathbb{T}$ occur as
$\sigma_{\mathrm{e}}(T)$ for a bounded normal operator $T$ which
admits a normalized orbit frame?
\end{problem}

The answer is determined by Haar (arclength) measure. We show (\prettyref{thm:4-5})
that a closed set $E\subset\mathbb{T}$ occurs in \prettyref{prob:1-1}
if and only if 
\[
m_{\mathbb{T}}(E)>0.
\]

Indeed, necessity holds under a weaker Bessel assumption. If the normalized
orbit $\left\{ x_{n}\right\} $ is Bessel with bound $B$ and $E_{g}$
is the outer spectral set of the cyclic pair, then
\[
Bm_{\mathbb{T}}(E_{g})\geq1
\]
and this estimate is sharp (Theorems \ref{thm:3-4}, \ref{cor:3-6}). 

For a normalized orbit frame the scalar spectral measure is purely
atomic, has no mass on the outer circle, and has no accumulation point
in the open disk (\prettyref{thm:2-3}). The outer spectral set is
then the essential spectrum. In particular, a normalized orbit of
a self-adjoint operator is never a Bessel sequence, since its outer
spectral set is contained in $\{-1,1\}$ (\prettyref{cor:3-7}). This
strengthens the non-frame result of \cite{MR3579135}.

Conversely, every closed set $E\subset\mathbb{T}$ of positive measure
is a solution to \prettyref{prob:1-1}. Moreover, there are absolute
constants $a,C>0$ such that the frame may be chosen with bounds (\prettyref{thm:4-5})
\[
a\leq A_{E}\leq B_{E}\leq\frac{C}{m_{\mathbb{T}}(E)}.
\]
Combined with the above lower bound, this gives the optimal dependence,
up to an absolute constant, of the upper frame bound on $E$.

The structural reduction used here is closely related to \cite{arXiv2511.15625}
on normal operators generating scalable iterative systems. In the
one-vector case, their results imply diagonal structure and spectral
localization inside the outer circle. We give a direct measure-theoretic
argument adapted to the present normalization and use the resulting
coefficient space for the essential spectrum classification.

\subsection*{Literature context}

Operator orbits have a history in linear dynamics, where one asks
whether a single orbit is cyclic or dense \cite{MR0241956,MR1111569}.
The frame problem imposes uniform upper and lower $\ell^{2}$ estimates
on the orbit coefficients. One may then ask which frames admit a representation
$\left(T^{n}g\right)_{n\geq0}$ with $T$ bounded, and what such a
representation implies about the frame and the operator \cite{MR3694620}.
Restrictions on orbit representations of classical structured frames
are studied in \cite{MR4208084}, and approximate representations
by suborbits are considered in \cite{MR4361596}. Continuous parameter
versions lead to frames generated by operator semigroups \cite{MR4943786}.
Invariant and hyperinvariant subspaces have also been studied for
dynamical frames \cite{MR4990268}.

For normal operators, the orbit can be described through the scalar
spectral measure. Bessel orbits admit a direct characterization of
this kind \cite{MR3582259}. With finitely many generators, frame
conditions can be expressed through Hardy space interpolation and
model spaces \cite{MR4198534}. In the diagonal one-vector setting
this leads to Carleson frames, whose redundancy properties have been
studied more recently in \cite{MR4735722,MR4961034}. These results
belong to the same function theoretic setting as the unnormalized
orbit problem described above, where the analytic space is fixed before
the orbit is considered.

The interpolation condition that occurs for diagonal normal orbits
comes from the classical interpolation theory of the disk. This includes
Carleson's characterization of interpolating sequences for $H^{\infty}$
and its $H^{2}$ counterpart, as well as related interpolation problems
for model and Bergman spaces \cite{MR117349,MR0133446,MR0827223,MR0654483,MR1223222,MR1758653}.
In all of these cases the analytic space is specified independently
of the interpolating sequence.

If $\rho$ is the pushforward of $\mu$ under $z\mapsto\left|z\right|^{2}$,
then $m_{n}=\int_{[0,1]}t^{n}d\rho(t)$, so $\left(m_{n}\right)_{n\geq0}$
is a Hausdorff moment sequence; see e.g., \cite{MR0184042}. In the
present problem the same measure $\mu$ determines the normalization
through $m_{n}$ and is also the measure of the target space $L^{2}\left(\mu\right)$.

\section{The measure model}\label{sec:2}

Let $\mu$ be a finite positive measure on $\overline{\mathbb{D}}$.
For $n\geq0$, define 
\begin{equation}
m_{n}:=\int_{\overline{\mathbb{D}}}|z|^{2n}d\mu(z),\quad u_{n}\left(z\right):=\frac{z^{n}}{\sqrt{m_{n}}}.\label{eq:2-1}
\end{equation}
Assume $m_{n}>0$ for every $n$.

We first normalize the outer radius of $supp\left(\mu\right)$. Let
\begin{equation}
r_{\mu}:=\inf\left\{ r\geq0:\mu\left(\left\{ z:|z|>r\right\} \right)=0\right\} >0.\label{eq:2-1a}
\end{equation}
Let $\nu\left(E\right):=\mu\left(r_{\mu}E\right)$, for all Borel
$E$. Then 
\[
\widetilde{m}_{n}=\int\left|z\right|^{2n}d\nu\left(z\right)=r^{-2n}_{\mu}m_{n},
\]
and there is a unitary operator 
\[
L^{2}(\mu)\rightarrow L^{2}(\nu),\quad z^{n}/\sqrt{m_{n}}\longmapsto z^{n}/\sqrt{\widetilde{m}_{n}}.
\]

Note that $r_{\nu}=1$, and the frame problem is unchanged by this
scaling. We may therefore assume from now on that 
\begin{equation}
r_{\mu}=1.\label{eq:2-2}
\end{equation}

Throughout the paper, all inner products are linear in the second
variable. 
\begin{lem}
We have 
\begin{equation}
\lim_{n\to\infty}m^{1/n}_{n}=1.\label{eq:2-3}
\end{equation}
\end{lem}

\begin{proof}
Let $0<r<1$. By \prettyref{eq:2-2}, $\mu\left(\left\{ z:|z|>r\right\} \right)>0$,
and so 
\begin{gather*}
m_{n}\geq\int_{\left\{ \left|z\right|>r\right\} }\left|z\right|^{2n}d\mu\left(z\right)\geq r^{2n}\mu\left(\left\{ z:|z|>r\right\} \right)\\
\Downarrow\\
\liminf_{n\rightarrow\infty}m^{1/n}_{n}\geq r^{2}.
\end{gather*}
Also, 
\[
m_{n}\leq\mu(\overline{\mathbb{D}})=m_{0}\Rightarrow\limsup_{n\rightarrow\infty}m^{1/n}_{n}\leq1.
\]
Letting $r\uparrow1$ gives \prettyref{eq:2-3}. 
\end{proof}

Let 
\begin{equation}
\mathcal{H}_{0}:=\left\{ \sum\nolimits_{n\geq0}c_{n}z^{n}:\sum\nolimits_{n\geq0}\left|c_{n}\right|^{2}m_{n}<\infty\right\} \label{eq:b-5}
\end{equation}
with norm 
\begin{equation}
\left\Vert \sum\nolimits_{n\geq0}c_{n}z^{n}\right\Vert ^{2}_{\mathcal{H}_{0}}:=\sum\nolimits_{n\geq0}\left|c_{n}\right|^{2}m_{n}.\label{eq:b-6}
\end{equation}

\begin{lem}
\label{lem:2-2}$\mathcal{H}_{0}$ consists of analytic functions
in $\mathbb{D}$. 
\end{lem}

\begin{proof}
Fix $0<r<1$. By \prettyref{eq:2-3}, 
\[
C:=\sum\nolimits_{n\geq0}\frac{r^{2n}}{m_{n}}<\infty.
\]
For $M>N$ and $\left|z\right|\leq r$, by Cauchy-Schwarz, 
\[
\left|\sum\nolimits^{M}_{n=N}c_{n}z^{n}\right|\leq\left(\sum\nolimits^{M}_{n=N}\left|c_{n}\right|^{2}m_{n}\right)^{1/2}\left(\sum\nolimits^{M}_{n=N}\frac{r^{2n}}{m_{n}}\right)^{1/2}.
\]
The first factor tends to zero with $N$, and the second is bounded
by $C^{1/2}$. Thus the series converges uniformly on every closed
disk contained in $\mathbb{D}$. In particular, convergence in $\mathcal{H}_{0}$
implies locally uniform convergence in $\mathbb{D}$. 
\end{proof}

By definition, $\left\{ u_{n}\right\} $ forms an orthonormal basis
(ONB) for $\mathcal{H}_{0}$, so every $f\in\mathcal{H}_{0}$ has
the representation 
\begin{equation}
f\left(z\right)=\sum\nolimits_{n\geq0}c_{n}z^{n},\quad\left\Vert f\right\Vert ^{2}_{\mathcal{H}_{0}}=\sum\nolimits_{n\geq0}\left|c_{n}\right|^{2}m_{n}.\label{eq:2-5}
\end{equation}
Equivalently, the map 
\begin{equation}
\ell^{2}\longrightarrow\mathcal{H}_{0},\quad\left(a_{n}\right)_{n\geq0}\longmapsto\sum_{n\geq0}a_{n}u_{n}\label{eq:2-6}
\end{equation}
is unitary.

Moreover, $\mathcal{H}_{0}$ is a reproducing kernel Hilbert space
(RKHS) on $\mathbb{D}$ with kernel 
\begin{equation}
K\left(z,w\right)=\sum\nolimits_{n\geq0}u_{n}\left(z\right)\overline{u_{n}\left(w\right)}=\sum\nolimits_{n\geq0}\frac{\left(z\overline{w}\right)^{n}}{m_{n}}.\label{eq:b-7}
\end{equation}
Thus, 
\[
f\left(w\right)=\left\langle K_{w},f\right\rangle _{\mathcal{H}_{0}},\quad\forall f\in\mathcal{H}_{0},
\]
where $K_{w}:=K\left(\cdot,w\right)$. 

Let $\left(u_{n}\right)_{n\geq0}$ be as in \prettyref{eq:2-1}. Suppose
it is a frame for $L^{2}\left(\mu\right)$ with frame bounds $A$
and $B$. Consider the analysis operator of $\left(u_{n}\right)$.
Using \prettyref{eq:2-6}, we get
\begin{equation}
J:L^{2}\left(\mu\right)\longrightarrow\ell^{2}\simeq\mathcal{H}_{0},\quad f\longmapsto\sum\left\langle u_{n},f\right\rangle u_{n}\label{eq:2-8}
\end{equation}
with adjoint/synthesis 
\begin{equation}
J^{*}:\mathcal{H}_{0}\longrightarrow L^{2}\left(\mu\right),\quad J^{*}\left(\sum a_{n}u_{n}\right)=\sum a_{n}u_{n}\label{eq:2-9}
\end{equation}
and frame bounds 
\[
AI\leq J^{*}J\leq BI.
\]

In \prettyref{eq:2-9}, the same coefficient expansion is interpreted
in two different spaces, as an element of $\mathcal{H}_{0}$ in the
domain and as an $L^{2}(\mu)$-convergent series in the range. For
an analytic polynomial $p$ the expansion is finite, so the two interpretations
give the same function and 
\begin{equation}
J^{*}p=p.\label{eq:2-10}
\end{equation}

The conclusion of the next theorem is essentially contained in the
one-generator results of \cite{arXiv2511.15625}. After normalizing
the outer spectral radius to one, those results give a pure point
spectral measure supported in $\mathbb{D}$, with no accumulation
point in $\mathbb{D}$. In the present model, however, the argument
becomes considerably simpler. A single nonzero function in $\ker J^{*}$
rules out boundary mass and then forces the remaining measure onto
its zero set. 
\begin{thm}
\label{thm:2-3} Suppose that $\left(u_{n}\right)_{n\geq0}$ is a
frame for $L^{2}(\mu)$. Then $\mu(\mathbb{T})=0$, and there are
distinct points $\lambda_{j}\in\mathbb{D}$ and positive numbers $w_{j}$
such that 
\begin{equation}
\mu=\sum w_{j}\delta_{\lambda_{j}}.\label{eq:2-11}
\end{equation}
The set $\left\{ \lambda_{j}\right\} $ has no accumulation point
in $\mathbb{D}$. 
\end{thm}

\begin{proof}
Consider 
\[
\xymatrix{L^{2}\left(\mu\right)\ar@/^{1.1pc}/[rr]^{J} &  & \mathcal{H}_{0}\ar@/^{1.1pc}/[ll]^{J^{*}}}
\]
as in \prettyref{eq:2-8}--\prettyref{eq:2-9}. We first show that
\[
\ker J^{*}\neq\left\{ 0\right\} .
\]
Assume otherwise. Since $\left(u_{n}\right)$ is a frame, $J^{*}$
is onto and hence invertible. Recall that $\left(u_{n}\right)$ is
an ONB for $\mathcal{H}_{0}$, and by \prettyref{eq:2-10}, 
\[
J^{*}u_{n}=u_{n}.
\]
Thus $\left(u_{n}\right)$ is a Riesz basis for $L^{2}\left(\mu\right)$.

Let $N$ be multiplication by $z$ on $L^{2}\left(\mu\right)$, 
\begin{equation}
f=\sum a_{n}u_{n}\longmapsto Nf=\sum a_{n}\sqrt{\frac{m_{n+1}}{m_{n}}}u_{n+1}.\label{eq:2-12}
\end{equation}
The expansion of $f$ in \prettyref{eq:2-12} is unique, as $\left(u_{n}\right)$
is a Riesz basis. But all the weights in \prettyref{eq:2-12} are
positive, thus $\ker N=\left\{ 0\right\} $, and 
\[
\overline{NL^{2}\left(\mu\right)}=\overline{span}\left\{ u_{n}:n\geq1\right\} \subsetneq L^{2}\left(\mu\right).
\]
Also, $N$ is normal, so $\ker N^{*}=\ker N=\left\{ 0\right\} $,
therefore 
\[
\overline{NL^{2}\left(\mu\right)}=\left(\ker N^{*}\right)^{\perp}=L^{2}\left(\mu\right),
\]
a contradiction. 

We next show that $\mu\left(\mathbb{T}\right)=0$. Suppose
\[
\sigma:=\mu|_{\mathbb{T}}\neq0,\quad b:=\sigma\left(\mathbb{T}\right)>0.
\]
We identify $L^{2}\left(\sigma\right)$ with the closed subspace of
$L^{2}\left(\mu\right)$ consisting of functions supported on $\mathbb{T}$.
Since $\left(u_{n}\right)$ is a frame for $L^{2}\left(\mu\right)$,
its restrictions 
\[
\left(z^{n}/\sqrt{m_{n}}\right)_{n\geq0}
\]
form a frame for $L^{2}\left(\sigma\right)$. Indeed, for $g\in L^{2}\left(\sigma\right)$,
\[
\left\langle u_{n},g\right\rangle _{L^{2}\left(\mu\right)}=\left\langle z^{n}/\sqrt{m_{n}},g\right\rangle _{L^{2}\left(\sigma\right)}.
\]

Note that $b\leq m_{n}\leq m_{0}$ for every $n$. It follows that
\[
\left(1,z,z^{2},\ldots\right)
\]
is also a frame for $L^{2}\left(\sigma\right)$. Let $M>0$ be a lower
frame bound. Let $0\neq g\in L^{2}\left(\sigma\right)$, then 
\begin{align*}
M\left\Vert g\right\Vert ^{2}_{L^{2}\left(\sigma\right)}=M\left\Vert \overline{z}^{k}g\right\Vert _{L^{2}\left(\sigma\right)} & \leq\sum_{n\geq0}\left|\left\langle z^{n},\overline{z}^{k}g\right\rangle _{L^{2}\left(\sigma\right)}\right|^{2}\\
 & =\sum_{n\geq k}\left|\left\langle z^{n},g\right\rangle _{L^{2}\left(\sigma\right)}\right|^{2}\xrightarrow[\;k\rightarrow\infty\;]{}0,
\end{align*}
which is a contradiction. Thus $\mu$ has no mass on $\mathbb{T}$.

Choose $0\neq f\in\ker J^{*}$, and let $p_{k}$ be polynomials converging
to $f$ in $\mathcal{H}_{0}$. Fix $0<r<1$. By \prettyref{lem:2-2},
$p_{k}\rightarrow f$ uniformly on 
\[
D_{r}=\left\{ z:\left|z\right|\leq r\right\} .
\]
Since $\mu$ is finite, it follows that $p_{k}\rightarrow f$ in $L^{2}(\mu|_{D_{r}})$.
On the other hand, $J^{*}p_{k}\rightarrow J^{*}f=0$ in $L^{2}\left(\mu\right)$.
Since $J^{*}p_{k}=p_{k}$ by \prettyref{eq:2-10}, we also have $p_{k}\rightarrow0$
in $L^{2}(\mu|_{D_{r}})$. Thus $f=0$ $\mu$-a.e. on $D_{r}$. 

Choose $r_{j}\uparrow1$ with 
\[
\mathbb{D}=\bigcup_{j\geq1}D_{r_{j}}.
\]
Then $f=0$ $\mu$-a.e. on $\mathbb{D}$. Since $\mu\left(\mathbb{T}\right)=0$,
the measure $\mu$ is concentrated on 
\[
Z\left(f\right)=\left\{ z\in\mathbb{D}:f\left(z\right)=0\right\} .
\]
The function $f$ is nonzero and analytic in $\mathbb{D}$, so $Z\left(f\right)$
is countable and has no accumulation point in $\mathbb{D}$. Therefore
$\mu$ is of the form \prettyref{eq:2-11}, and the set $\left\{ \lambda_{j}\right\} $
has no accumulation point in $\mathbb{D}$.
\end{proof}

\begin{cor}
\label{cor:2-4}Suppose that $\left(u_{n}\right)_{n\geq0}$ is a frame
for $L^{2}\left(\mu\right)$, and write $\mu=\sum_{j}w_{j}\delta_{\lambda_{j}}$
as in \prettyref{thm:2-3}. Then 
\begin{equation}
\ker J^{*}=\left\{ f\in\mathcal{H}_{0}:f\left(\lambda_{j}\right)=0,\:\forall j\right\} .\label{eq:2-13}
\end{equation}
\end{cor}

\begin{proof}
In fact, for every $f\in\mathcal{H}_{0}$, 
\[
\left(J^{*}f\right)\left(\lambda_{j}\right)=f\left(\lambda_{j}\right)
\]
for every $j$. 

To see this, let $e_{j}=w^{-1/2}_{j}\chi_{\left\{ \lambda_{j}\right\} }\in L^{2}\left(\mu\right)$.
Then $\left(e_{j}\right)$ is the standard ONB of $L^{2}\left(\mu\right)$.
From the definition of $J$ and \prettyref{eq:b-7}, we have 
\[
Je_{j}=\sum_{n\geq0}\left\langle u_{n},e_{j}\right\rangle _{L^{2}\left(\mu\right)}u_{n}=\sqrt{w_{j}}\sum_{n\geq0}\overline{u_{n}\left(\lambda_{j}\right)}u_{n}=\sqrt{w_{j}}K_{\lambda_{j}}.
\]
Now for $f\in\mathcal{H}_{0}$, 
\begin{align*}
\left(J^{*}f\right)\left(\lambda_{j}\right) & =\frac{1}{\sqrt{w_{j}}}\left\langle e_{j},J^{*}f\right\rangle _{L^{2}\left(\mu\right)}\\
 & =\frac{1}{\sqrt{w_{j}}}\left\langle Je_{j},f\right\rangle _{\mathcal{H}_{0}}=\left\langle K_{\lambda_{j}},f\right\rangle _{\mathcal{H}_{0}}=f\left(\lambda_{j}\right).
\end{align*}

Thus 
\[
\left\Vert J^{*}f\right\Vert ^{2}_{L^{2}\left(\mu\right)}=\sum_{j}w_{j}\left|f\left(\lambda_{j}\right)\right|^{2}.
\]
Since $w_{j}>0$, \prettyref{eq:2-13} follows. 
\end{proof}

For an atomic measure, the frame condition can also be viewed in $\mathcal{H}_{0}$.
\begin{thm}
\label{thm:2-5} Let $\mu$ be an atomic measure of the form \prettyref{eq:2-11},
with distinct points $\lambda_{j}\in\mathbb{D}$. For every atom $\lambda_{j}$,
define 
\begin{equation}
h_{j}=\sqrt{w_{j}}K_{\lambda_{j}}.\label{eq:2-14}
\end{equation}
Then the normalized monomials $\left(u_{n}\right)_{n\geq0}$ form
a frame for $L^{2}(\mu)$ with bounds $A$ and $B$ if and only if
$\left(h_{j}\right)$ is a Riesz sequence in $\mathcal{H}_{0}$ with
the same bounds. 
\end{thm}

\begin{proof}
First assume $\left(u_{n}\right)$ is a frame. Let $e_{j}=w^{-1/2}_{j}\chi_{\{\lambda_{j}\}}$,
an ONB of $L^{2}(\mu)$. By the proof of \prettyref{cor:2-4}, 
\[
Je_{j}=h_{j}.
\]
For every finitely supported $\left(c_{j}\right)$, 
\[
A\sum|c_{j}|^{2}\leq\left\Vert J\left(\sum c_{j}e_{j}\right)\right\Vert ^{2}_{\mathcal{H}_{0}}=\left\Vert \sum c_{j}h_{j}\right\Vert ^{2}_{\mathcal{H}_{0}}\leq B\sum|c_{j}|^{2}.
\]
So $\left(h_{j}\right)$ is a Riesz sequence in $\mathcal{H}_{0}$
with bounds $A$ and $B$.

Conversely, let $g=\sum_{j}c_{j}e_{j}$, where $\left(c_{j}\right)$
is finitely supported. Since $h_{j}=\sqrt{w_{j}}K_{\lambda_{j}}$,
\[
\left\langle u_{n},\sum c_{j}h_{j}\right\rangle _{\mathcal{H}_{0}}=\sum c_{j}\sqrt{w_{j}}\overline{u_{n}\left(\lambda_{j}\right)}=\left\langle u_{n},g\right\rangle _{L^{2}(\mu)}.
\]
Since $\left(u_{n}\right)$ is an ONB of $\mathcal{H}_{0}$, 
\[
\sum\left|\left\langle u_{n},g\right\rangle _{L^{2}(\mu)}\right|^{2}=\left\Vert \sum c_{j}h_{j}\right\Vert ^{2}_{\mathcal{H}_{0}}.
\]
The Riesz inequalities for $\left(h_{j}\right)$ now give 
\begin{gather*}
A\sum|c_{j}|^{2}\leq\left\Vert \sum c_{j}h_{j}\right\Vert ^{2}_{\mathcal{H}_{0}}\leq B\sum|c_{j}|^{2}\\
\Updownarrow\\
A\left\Vert g\right\Vert ^{2}_{L^{2}(\mu)}\leq\sum\left|\left\langle u_{n},g\right\rangle _{L^{2}(\mu)}\right|^{2}\leq B\left\Vert g\right\Vert ^{2}_{L^{2}(\mu)}.
\end{gather*}
This extends from $span\left\{ e_{j}\right\} $ to $L^{2}(\mu)$.
\end{proof}

\section{The outer spectral set}\label{sec:3}

We now turn to the spectral distinction between self-adjoint and normal
operators. The first step uses only an upper frame bound for the normalized
orbit. Throughout this section, $m_{\mathbb{T}}$ denotes normalized
Haar measure on $\mathbb{T}$.

Let $\mu$ be a finite positive measure on $\overline{\mathbb{D}}$
whose outer radius is one. We do not assume that $\mu$ is atomic.
As in \prettyref{sec:2}, set 
\[
m_{n}=\int_{\overline{\mathbb{D}}}|z|^{2n}d\mu\left(z\right),\quad u_{n}\left(z\right)=\frac{z^{n}}{\sqrt{m_{n}}},
\]
where $m_{n}>0$ for every $n$.

Define the outer spectral set 
\begin{equation}
E_{\mu}:=supp\left(\mu\right)\cap\mathbb{T}.\label{eq:c-1}
\end{equation}

\begin{rem*}
We need to separate two boundary statements that arise in the frame
case. By \prettyref{thm:2-3}, if $\left\{ u_{n}\right\} $ forms
a frame, then $\mu\left(\mathbb{T}\right)=0$, no mass on the outer
circle. For $\mu=\sum\nolimits_{j}w_{j}\delta_{\lambda_{j}}$, $\lambda_{j}\in\mathbb{D}$,
we have 
\[
supp\left(\mu\right)=\overline{\left\{ \lambda_{j}:j\geq1\right\} }
\]
so $E_{\mu}$ consists of the boundary accumulation points of the
atoms.

In particular, $m_{\mathbb{T}}\left(E_{\mu}\right)$ measures the
size, with respect to Haar measure, of the boundary set approached
by the interior atoms. It is possible to have 
\[
\mu\left(\mathbb{T}\right)=0\quad\text{and}\quad m_{\mathbb{T}}\left(E_{\mu}\right)>0.
\]

This distinction has a direct operator interpretation. Recall that
the essential spectrum of a bounded operator $T$ is 
\[
\sigma_{\mathrm{e}}\left(T\right)=\left\{ \lambda\in\mathbb{C}:T-\lambda I\text{ is not Fredholm}\right\} .
\]
For the atomic measure in \prettyref{thm:2-3}, multiplication by
$z$ on $L^{2}\left(\mu\right)$ is diagonal with distinct eigenvalues
$\lambda_{j}\in\mathbb{D}$, no accumulation point in $\mathbb{D}$.
Therefore 
\[
\sigma_{\mathrm{e}}\left(M_{z}\right)=supp\left(\mu\right)\cap\mathbb{T}.
\]
See also \prettyref{eq:3-8a} below for the corresponding operator
model. 
\end{rem*}
The first observation is a consequence of the moment structure of
$\left(m_{n}\right)$. 
\begin{lem}
\label{lem:3-1} For every fixed integer $k\geq0$, 
\[
\frac{m_{n+k}}{m_{n}}\xrightarrow[\;n\rightarrow\infty\;]{}1.
\]
\end{lem}

\begin{proof}
Let $\rho$ be the pushforward of $\mu$ under $z\mapsto|z|^{2}$,
so that 
\[
m_{n}=\int_{[0,1]}t^{n}d\rho\left(t\right).
\]
Cauchy-Schwarz gives 
\[
m^{2}_{n}\leq m_{n-1}m_{n+1}\Longleftrightarrow\frac{m_{n}}{m_{n-1}}\leq\frac{m_{n+1}}{m_{n}},
\]
i.e., $m_{n+1}/m_{n}$ is increasing, bounded above by $1$. Let 
\[
a=\lim_{n\to\infty}\frac{m_{n+1}}{m_{n}}.
\]
Then 
\[
\left(\frac{m_{n}}{m_{0}}\right)^{1/n}=\underset{\text{geometric mean}}{\underbrace{\left(\prod\nolimits^{n-1}_{j=0}\frac{m_{j+1}}{m_{j}}\right)^{1/n}}}\xrightarrow[\;n\rightarrow\infty\;]{}a.
\]
But $m^{1/n}_{n}\to1$ by \prettyref{eq:2-3}, so $a=1$. Then, for
$k\geq1$, 
\[
\frac{m_{n+k}}{m_{n}}=\prod^{k-1}_{j=0}\frac{m_{n+j+1}}{m_{n+j}}\xrightarrow[\;n\rightarrow\infty\;]{}1.
\]
\end{proof}

For $n\geq0$ define a probability measure on $\overline{\mathbb{D}}$
by 
\begin{equation}
d\eta_{n}\left(z\right):=\frac{|z|^{2n}}{m_{n}}d\mu\left(z\right).\label{eq:3-1a}
\end{equation}
As $n\rightarrow\infty$, these measures concentrate on $E_{\mu}$.
More precisely, the following holds.
\begin{lem}
\label{lem:3-2} Every weak limit of a subsequence of $\left(\eta_{n}\right)$
is supported on $E_{\mu}$. 
\end{lem}

\begin{proof}
Fix $0<r<1$. From \prettyref{eq:2-3}, 
\[
\eta_{n}\left(\left\{ |z|<r\right\} \right)\leq\frac{r^{2n}\mu\left(\overline{\mathbb{D}}\right)}{m_{n}}\longrightarrow0.
\]
Thus every weak limit is supported on $\mathbb{T}$. 

Since all $\eta_{n}$ vanish outside $supp\left(\mu\right)$, the
same is true of any weak limit. Its support is therefore contained
in $supp\left(\mu\right)\cap\mathbb{T}=E_{\mu}$. 
\end{proof}

Now we impose an upper frame bound for $\left(u_{n}\right)$, and
this gives a lower bound for $m_{\mathbb{T}}\left(E_{\mu}\right)$.
\begin{thm}
\label{thm:3-3} Suppose that $\left(u_{n}\right)_{n\geq0}$ is a
Bessel sequence in $L^{2}\left(\mu\right)$ with Bessel bound $B$.
Then 
\begin{equation}
Bm_{\mathbb{T}}\left(E_{\mu}\right)\geq1.\label{eq:3-1}
\end{equation}
In particular, $m_{\mathbb{T}}\left(E_{\mu}\right)>0$. 
\end{thm}

\begin{proof}
By assumption, 
\begin{equation}
\left\Vert \sum a_{j}u_{j}\right\Vert ^{2}\leq B\sum\left|a_{j}\right|^{2}\label{eq:3-2}
\end{equation}
for every finitely supported sequence $\left(a_{j}\right)$.

Choose a subsequence $n_{j}$ such that $\eta_{n_{j}}\xrightarrow{w}\sigma$,
a probability measure supported on $E_{\mu}$ (\prettyref{lem:3-2}). 

Let $p\left(z\right)=\sum^{N}_{k=0}c_{k}z^{k}$. Using \prettyref{eq:3-2},
we get 
\begin{alignat*}{1}
\int_{\overline{\mathbb{D}}}\left|z\right|^{2n}\left|p\left(z\right)\right|^{2}d\mu\left(z\right) & =\left\Vert \sum\nolimits^{N}_{k=0}c_{k}z^{n+k}\right\Vert ^{2}_{L^{2}\left(\mu\right)}\\
 & =\left\Vert \sum\nolimits^{N}_{k=0}c_{k}\sqrt{m_{n+k}}u_{n+k}\right\Vert ^{2}_{L^{2}\left(\mu\right)}\\
 & \leq B\sum\nolimits^{N}_{k=0}|c_{k}|^{2}m_{n+k}.
\end{alignat*}
Dividing by $m_{n}$ gives 
\[
\int_{\overline{\mathbb{D}}}\left|p\left(z\right)\right|^{2}d\eta_{n}\left(z\right)\leq B\sum^{N}_{k=0}\left|c_{k}\right|^{2}\frac{m_{n+k}}{m_{n}}.
\]
Now let $n=n_{j}$ and pass to the limit. By \prettyref{lem:3-1},
\begin{equation}
\int_{\mathbb{T}}\left|p\right|^{2}d\sigma\leq B\sum^{N}_{k=0}\left|c_{k}\right|^{2}=B\int_{\mathbb{T}}\left|p\right|^{2}dm_{\mathbb{T}}.\label{eq:3-3}
\end{equation}

By Fejér-Riesz, \prettyref{eq:3-3} holds with $\left|p\right|^{2}$
replaced by any nonnegative trigonometric polynomial. Taking Fejér
means gives
\begin{gather*}
\int_{\mathbb{T}}fd\sigma\leq B\int_{\mathbb{T}}fdm_{\mathbb{T}},\quad\forall f\in C\left(\mathbb{T}\right),\:f\geq0\\
\Updownarrow\\
\sigma\leq Bm_{\mathbb{T}}.
\end{gather*}
Since $\sigma$ is a probability measure supported on $E_{\mu}$,
\[
1=\sigma\left(E_{\mu}\right)\leq Bm_{\mathbb{T}}\left(E_{\mu}\right),
\]
which is \prettyref{eq:3-1}.
\end{proof}

Next we show the upper frame bound $B$ in \prettyref{eq:3-1} is
sharp. 
\begin{thm}
\label{thm:3-4} Let $F\subset\mathbb{T}$ be closed and $m_{\mathbb{T}}\left(F\right)>0$.
There exists a finite positive measure $\mu$ supported on $F$ such
that $\left\{ u_{n}\right\} $ is Bessel in $L^{2}\left(\mu\right)$
with optimal Bessel bound 
\[
\frac{1}{m_{\mathbb{T}}\left(F\right)}.
\]
\end{thm}

\begin{proof}
Choose 
\[
d\mu=\frac{1}{\alpha}\chi_{F}dm_{\mathbb{T}},\quad\alpha:=m_{\mathbb{T}}\left(F\right),
\]
i.e., normalized Haar measure restricted to $F$.

Then $m_{n}=\int\left|z\right|^{2n}d\mu=\int_{F}1d\mu=1$, $\forall n$,
and $u_{n}\left(z\right)=z^{n}$. For $f\in L^{2}\left(\mu\right)$,
\begin{align*}
\sum_{n\geq0}\left|\left\langle u_{n},f\right\rangle _{L^{2}\left(\mu\right)}\right|^{2} & =\sum_{n\geq0}\left|\int\overline{z}^{n}f\left(z\right)d\mu\left(z\right)\right|^{2}\\
 & =\frac{1}{\alpha^{2}}\sum_{n\geq0}\left|\int\overline{z}^{n}\chi_{F}f\left(z\right)dm_{\mathbb{T}}\left(z\right)\right|^{2}\\
 & =\frac{1}{\alpha^{2}}\sum_{n\geq0}\left|\left\langle z^{n},\chi_{F}f\right\rangle _{L^{2}\left(m_{\mathbb{T}}\right)}\right|^{2}\\
 & \leq\frac{1}{\alpha^{2}}\int_{F}\left|f\right|^{2}dm_{\mathbb{T}}=\frac{1}{\alpha}\left\Vert f\right\Vert ^{2}_{L^{2}(\mu)}.
\end{align*}
Thus $1/\alpha$ is a Bessel bound. 

By \prettyref{thm:3-3}, any Bessel bound $B$ must satisfy 
\[
B\geq\frac{1}{m_{\mathbb{T}}\left(E_{\mu}\right)}.
\]
We claim that 
\begin{equation}
m_{\mathbb{T}}\left(E_{\mu}\right)=m_{\mathbb{T}}\left(F\right)=\alpha\label{eq:3-4}
\end{equation}
and so 
\[
B\geq\frac{1}{\alpha}.
\]
Therefore the Bessel bound $1/\alpha$ obtained above is optimal.

Indeed, since $\mu$ is supported on $F\subset\mathbb{T}$, 
\[
E_{\mu}=supp\left(\mu\right)\subset F.
\]
The support of $\mu$ has full $\mu$-measure, so 
\[
0=\mu\left(F\setminus E_{\mu}\right)=\frac{1}{\alpha}m_{\mathbb{T}}\left(F\setminus E_{\mu}\right).
\]
This gives \prettyref{eq:3-4}. 
\end{proof}

We note that equality in \prettyref{eq:3-1} also determines the limiting
angular measure. 
\begin{thm}
\label{thm:3-5} Under the hypotheses of \prettyref{thm:3-3}, suppose
that 
\[
Bm_{\mathbb{T}}\left(E_{\mu}\right)=1.
\]
Let $\eta_{n}$ be as in \prettyref{eq:3-1a}. Then 
\begin{equation}
\eta_{n}\longrightarrow\frac{1}{m_{\mathbb{T}}\left(E_{\mu}\right)}\chi_{E_{\mu}}dm_{\mathbb{T}}\label{eq:3-9}
\end{equation}
weakly on $\overline{\mathbb{D}}$. 
\end{thm}

\begin{proof}
Let $\sigma$ be any weak limit of a subsequence of $\left(\eta_{n}\right)$.
The proof of \prettyref{thm:3-3} gives 
\[
\sigma\leq Bm_{\mathbb{T}},\quad\sigma\left(E_{\mu}\right)=1.
\]
Since $\sigma$ is supported on $E_{\mu}$, 
\[
\sigma\leq B\chi_{E_{\mu}}m_{\mathbb{T}}.
\]
By assumption, 
\[
Bm_{\mathbb{T}}\left(E_{\mu}\right)=1,
\]
so $\sigma$ and $B\chi_{E_{\mu}}m_{\mathbb{T}}$ have the same total
mass. Hence 
\[
\sigma=B\chi_{E_{\mu}}m_{\mathbb{T}}=\frac{1}{m_{\mathbb{T}}\left(E_{\mu}\right)}\chi_{E_{\mu}}dm_{\mathbb{T}}.
\]
Thus every convergent subsequence has the same weak limit. Since the
probability measures on $\overline{\mathbb{D}}$ form a compact space
in the weak topology, the full sequence converges. 
\end{proof}

\subsection{Operator consequences}

Let $T$ be a normal operator on $\mathcal{H}$, and let $0\neq g\in\mathcal{H}$.
Let $\mathcal{H}_{g}$ be the smallest reducing subspace for $T$
containing $g$, and $\mu_{g}$ the corresponding scalar spectral
measure. By the spectral theorem, 
\begin{equation}
T|_{\mathcal{H}_{g}}\sim M_{z}=\text{multiplication by \ensuremath{z} on }L^{2}\left(\mu_{g}\right).\label{eq:3-8a}
\end{equation}

Assume that $T^{n}g\neq0$ for every $n$. Then $r_{\mu_{g}}>0$,
where $r_{\mu_{g}}$ is the outer radius defined in \prettyref{eq:2-1a}.
Replacing $T$ by $r^{-1}_{\mu_{g}}T$ does not change the normalized
orbit. Keeping the same notation for $T$ and $\mu_{g}$ after this
rescaling, we may therefore assume that $r_{\mu_{g}}=1$.

Recall that 
\[
E_{\mu_{g}}=supp\left(\mu_{g}\right)\cap\mathbb{T}.
\]

Under the above assumptions, we have:
\begin{cor}
\label{cor:3-6} Suppose that $\left(T^{n}g/\left\Vert T^{n}g\right\Vert \right)_{n\geq0}$
is a Bessel sequence in $\mathcal{H}$ with bound $B$. Then 
\begin{equation}
Bm_{\mathbb{T}}\left(E_{\mu_{g}}\right)\geq1.\label{eq:3-11}
\end{equation}
\end{cor}

\begin{proof}
Under the cyclic spectral representation, the normalized orbit corresponds
to the normalized monomials in $L^{2}\left(\mu_{g}\right)$. Since
$r_{\mu_{g}}=1$, the conclusion follows from \prettyref{thm:3-3}. 
\end{proof}

\begin{cor}
\label{cor:3-7} If, in addition, $T$ is self-adjoint, then $\left(T^{n}g/\left\Vert T^{n}g\right\Vert \right)_{n\geq0}$
cannot be a Bessel sequence. 
\end{cor}

\begin{proof}
Since $T$ is self-adjoint and $r_{\mu_{g}}=1$, $supp\left(\mu_{g}\right)\subset[-1,1]$.
Hence 
\[
E_{\mu_{g}}=supp\left(\mu_{g}\right)\cap\mathbb{T}\subset\{-1,1\},
\]
so $m_{\mathbb{T}}\left(E_{\mu_{g}}\right)=0$. This contradicts \prettyref{eq:3-11}
if the normalized orbit is Bessel. 
\end{proof}

\begin{rem}
These corollaries are related to several results on operator orbits
and their normalizations. 

For an unnormalized orbit $\left(T^{n}g\right)_{n\geq0}$, Besselness
for a normal operator was characterized in terms of the scalar spectral
measure in \cite{MR3582259}. In the self-adjoint case, the characterization
takes the form 
\[
\left(T^{n}g\right)_{n\geq0}\text{ Bessel }\Longleftrightarrow\mu_{g}\left(\left\{ t:\left|t\right|>1-\varepsilon\right\} \right)=O\left(\varepsilon\right)\Longleftrightarrow\left\langle g,T^{n}g\right\rangle =O\left(n^{-1}\right).
\]
The equivalence follows from the corresponding Hankel matrix criterion;
see \cite[Theorems 3.4, 3.5]{MR3582259}. Thus an unnormalized self-adjoint
orbit may be Bessel. 

The condition above depends on the distribution of the scalar spectral
measure near the outer radius. For the normalized orbit, this dependence
is carried by the probability measures 
\[
\frac{\left|z\right|^{2n}}{m_{n}}d\mu_{g}\left(z\right).
\]
Their weak limits are supported on $E_{\mu_{g}}$. Together with the
Bessel inequality, this gives \prettyref{eq:3-11}.

Normalized iterative systems were considered in \cite{MR3613395,MR3579135}.
The first gives a non-frame result in a self-adjoint pure point setting.
The second removes the pure point assumption and allows a countable
family of generating vectors. Its proof uses the Feichtinger theorem
and the Müntz-Szász theorem. Although the result is stated for frames,
the same argument with the Bessel form of the Feichtinger theorem
also rules out Besselness. Therefore, \prettyref{cor:3-7} gives a
short spectral proof, for a single orbit, of a conclusion implicit
in that argument.

Bessel-normalizable and frame-normalizable sequences were studied
more generally in \cite{arXiv2308.13071}. For iterative systems,
several non-frame and non-Bessel results are obtained under additional
assumptions on the unnormalized system. \prettyref{cor:3-6} requires
no such assumption for a single normal orbit. Frame orbits of normal
operators have also been studied without normalization, in connection
with interpolating sequences and model spaces \cite{MR4093918,MR4198534}.
For general normal operators, normalized orbits may in fact be frames.
The contrary conjecture was stated in \cite{MR5071919} and disproved
in \cite{arXiv2606.20848}.
\end{rem}

\section{Realizing essential spectra}\label{sec:4}

Our target is to characterize which $E\subset\mathbb{T}$ can be $\sigma_{\mathrm{e}}\left(T\right)$
for some normal $T$ admitting a normalized orbit frame (\prettyref{prob:1-1}).
\prettyref{sec:3} shows that $m_{\mathbb{T}}\left(E\right)>0$ is
necessary. We now prove the converse.

The analytic part uses a fixed RKHS. Let $\mathcal{B}$ have kernel
\begin{equation}
B\left(z,w\right)=\frac{1}{\left(1-z\overline{w}\right)^{3}}=\sum_{n\geq0}b_{n}z^{n}\overline{w}^{n},\quad b_{n}=\frac{\left(n+1\right)\left(n+2\right)}{2}.\label{eq:4-1}
\end{equation}
For $w\in\mathbb{D}$, write 
\[
B_{w}=B\left(\cdot,w\right),\quad e_{w}=\left(1-\left|w\right|^{2}\right)^{3/2}B_{w}.
\]
Then $\left\Vert e_{w}\right\Vert _{\mathcal{B}}=1$.
\begin{lem}
\label{lem:4-1}For $z,w\in\mathbb{D}$, 
\begin{equation}
\left|\left\langle e_{z},e_{w}\right\rangle _{\mathcal{B}}\right|=\frac{1}{\cosh^{3}\left(d_{\mathbb{D}}\left(z,w\right)/2\right)},\label{eq:4-2}
\end{equation}
where $d_{\mathbb{D}}$ is the hyperbolic distance on $\mathbb{D}$
normalized by 
\[
\left|\frac{z-w}{1-\overline{w}z}\right|=\tanh\frac{d_{\mathbb{D}}\left(z,w\right)}{2}.
\]
\end{lem}

\begin{proof}
By the reproducing property, 
\begin{align*}
\left|\left\langle e_{z},e_{w}\right\rangle _{\mathcal{B}}\right| & =\left(1-\left|z\right|^{2}\right)^{3/2}\left(1-\left|w\right|^{2}\right)^{3/2}\cdot\left|B\left(z,w\right)\right|\\
 & =\left[\frac{\left(1-\left|z\right|^{2}\right)\left(1-\left|w\right|^{2}\right)}{\left|1-z\overline{w}\right|^{2}}\right]^{3/2}=\left(1-\left|\frac{z-w}{1-\overline{w}z}\right|^{2}\right)^{3/2}.
\end{align*}
Now \prettyref{eq:4-2} follows from the definition of $d_{\mathbb{D}}$. 
\end{proof}

The use of pseudohyperbolic separation in \prettyref{lem:4-2} below
is classical in Bergman interpolation; see e.g., \cite{MR0654483,MR1223222,MR1758653}.
For the kernel \prettyref{eq:4-1}, we give a direct proof that sufficiently
large hyperbolic separation gives the Riesz property. 

Recall that $\left(x_{j}\right)$ is a Riesz sequence in a Hilbert
space $H$ if there are $A,B>0$ such that

\[
A\sum\nolimits_{j}\left|c_{j}\right|^{2}\leq\left\Vert \sum\nolimits_{j}c_{j}x_{j}\right\Vert ^{2}_{H}\leq B\sum\nolimits_{j}\left|c_{j}\right|^{2}
\]
for every finitely supported sequence $\left(c_{j}\right)$ in $\mathbb{C}$.
\begin{lem}
\label{lem:4-2} There is an absolute constant $M>2$ such that if
$\Lambda\subset\mathbb{D}$ is a finite or countable with 
\[
d_{\mathbb{D}}\left(z,w\right)\geq M,\quad\forall z,w\in\Lambda,\:z\neq w,
\]
then $\left(e_{w}\right)_{w\in\Lambda}$ is a Riesz sequence in $\mathcal{B}$
with absolute Riesz bounds.
\end{lem}

\begin{proof}
Let $M>2$ for the moment, and suppose that $\Lambda$ is $M$-separated.
Fix $z\in\Lambda$. We will estimate 
\[
\sum_{w\in\Lambda\backslash\left\{ z\right\} }\left|\left\langle e_{z},e_{w}\right\rangle _{\mathcal{B}}\right|.
\]

For each $w\in\Lambda\backslash\left\{ z\right\} $, let 
\[
D_{w}=\left\{ \zeta\in\mathbb{D}:d_{\mathbb{D}}\left(\zeta,w\right)<1\right\} .
\]
These disks are pairwise disjoint.

We next compare the contribution of a point $w$ with the values on
the disk $D_{w}$. If $\zeta\in D_{w}$, then 
\[
d_{\mathbb{D}}\left(z,\zeta\right)\leq d_{\mathbb{D}}\left(z,w\right)+1.
\]
For $t\geq0$, 
\[
\cosh\frac{t+1}{2}\leq e^{1/2}\cosh\frac{t}{2}.
\]
It follows from \prettyref{eq:4-2} that, for every $\zeta\in D_{w}$,
\[
\left|\left\langle e_{z},e_{w}\right\rangle _{\mathcal{B}}\right|=\frac{1}{\cosh^{3}\left(d_{\mathbb{D}}\left(z,w\right)/2\right)}\leq\frac{e^{3/2}}{\cosh^{3}\left(d_{\mathbb{D}}\left(z,\zeta\right)/2\right)}.
\]

Let $dA_{h}$ denote hyperbolic area. Every hyperbolic disk of radius
one has the same positive area. Integrating the last inequality over
$D_{w}$ therefore gives 
\[
\left|\left\langle e_{z},e_{w}\right\rangle _{\mathcal{B}}\right|\leq C\int_{D_{w}}\frac{dA_{h}\left(\zeta\right)}{\cosh^{3}\left(d_{\mathbb{D}}\left(z,\zeta\right)/2\right)},
\]
where $C$ is an absolute constant.

We now sum over $w\neq z$. Since the disks $D_{w}$ are pairwise
disjoint, 
\[
\sum_{w\in\Lambda\backslash\left\{ z\right\} }\left|\left\langle e_{z},e_{w}\right\rangle _{\mathcal{B}}\right|\leq C\int_{\bigcup_{w\neq z}D_{w}}\frac{dA_{h}\left(\zeta\right)}{\cosh^{3}\left(d_{\mathbb{D}}\left(z,\zeta\right)/2\right)}.
\]

The region of integration is also easy to locate. If $\zeta\in D_{w}$,
with $w\neq z$, then 
\[
d_{\mathbb{D}}\left(z,\zeta\right)\geq d_{\mathbb{D}}\left(z,w\right)-d_{\mathbb{D}}\left(w,\zeta\right)>M-1.
\]
Thus, 
\[
\bigcup_{w\neq z}D_{w}\subset\left\{ \zeta\in\mathbb{D}:d_{\mathbb{D}}\left(z,\zeta\right)>M-1\right\} .
\]
It follows that 
\[
\sum_{w\in\Lambda\backslash\left\{ z\right\} }\left|\left\langle e_{z},e_{w}\right\rangle _{\mathcal{B}}\right|\leq C\int_{\left\{ d_{\mathbb{D}}(z,\zeta)>M-1\right\} }\frac{dA_{h}(\zeta)}{\cosh^{3}(d_{\mathbb{D}}(z,\zeta)/2)}.
\]

Use hyperbolic polar coordinates centered at $z$. Up to an absolute
constant, the last integral is given by 
\[
\int^{\infty}_{M-1}\frac{\sinh t}{\cosh^{3}\left(t/2\right)}dt=\frac{4}{\cosh\left(\left(M-1\right)/2\right)}.
\]
Thus 
\begin{equation}
\sup_{z\in\Lambda}\sum_{w\in\Lambda\backslash\left\{ z\right\} }\left|\left\langle e_{z},e_{w}\right\rangle _{\mathcal{B}}\right|\leq\frac{C}{\cosh\left(\left(M-1\right)/2\right)}.\label{eq:4-3}
\end{equation}

The right hand side tends to zero as $M$ tends to infinity. Choose
$M$ sufficiently large that it is at most $1/2$.

Let $G$ be the Gram matrix of $(e_{w})_{w\in\Lambda}$ and write
\[
G=I+R.
\]
By \prettyref{eq:4-3}, every absolute row sum of $R$ is at most
$1/2$. Since $R=R^{*}$, the same is true of the column sums. Schur
test then gives 
\[
\left\Vert R\right\Vert \leq\frac{1}{2}.
\]
So for every finitely supported family $(c_{w})$, 
\[
\frac{1}{2}\sum\nolimits_{w\in\Lambda}\left|c_{w}\right|^{2}\leq\left\Vert \sum\nolimits_{w\in\Lambda}c_{w}e_{w}\right\Vert ^{2}_{\mathcal{B}}\leq\frac{3}{2}\sum\nolimits_{w\in\Lambda}\left|c_{w}\right|^{2}.
\]
Thus $(e_{w})_{w\in\Lambda}$ is a Riesz sequence with absolute Riesz
bounds. 
\end{proof}

Set 
\[
\delta\left(z\right)=1-\left|z\right|^{2},\quad z\in\mathbb{D}.
\]
For a closed set $E\subset\mathbb{T}$, set 
\[
\Gamma_{E}=\left\{ r\zeta:\zeta\in E,\ \frac{1}{2}<r<1\right\} .
\]
For $\Lambda\subset\Gamma_{E}$, let
\[
F_{\Lambda}\left(t\right)=\sum_{\lambda\in\Lambda,\:\delta\left(\lambda\right)\leq t}\delta\left(\lambda\right)^{3},\quad0<t\leq1.
\]

We need a set $\Lambda$ with two properties: (a) it is sparse enough
for the normalized kernels to form a Riesz sequence; and (b) dense
enough near $E$ to give the required mass near the boundary. Hyperbolic
separation gives the first condition, maximality gives the second. 

At scale $t$, the number of points is of order $1/t$, and the weight
$\delta^{3}$ then gives total mass of order $t^{2}$. The next result
makes this precise.
\begin{prop}
\label{prop:4-3}Fix $M>0$. If $E\subset\mathbb{T}$ is closed and
$m_{\mathbb{T}}\left(E\right)>0$, then there is a countable $M$-separated
set $\Lambda\subset\Gamma_{E}$ with no accumulation point in $\mathbb{D}$
and with boundary accumulation set $E$. 

Moreover, 
\begin{equation}
c_{M}m_{\mathbb{T}}\left(E\right)t^{2}\leq F_{\Lambda}\left(t\right)\leq C_{M}t^{2},\quad0<t\leq1,\label{eq:4-4}
\end{equation}
where $c_{M},C_{M}>0$ depend only on $M$. 
\end{prop}

\begin{proof}
Choose $\Lambda$ maximal among the $M$-separated subsets of $\Gamma_{E}$
(Zorn's lemma). Since $\mathbb{D}$ is separable, $\Lambda$ is countable.
Every point of $\Gamma_{E}$ thus has hyperbolic distance less than
$M$ from some point of $\Lambda$.

\textbf{1.} Claim: For $z,w\in\Gamma_{E}$ with $d_{\mathbb{D}}\left(z,w\right)<M$,
there exist $a_{M},b_{M}>0$ such that 
\begin{gather}
a_{M}\delta\left(z\right)\leq\delta\left(w\right)\leq b_{M}\delta\left(z\right),\label{eq:d-5}\\
d_{\mathbb{T}}\left(z/\left|z\right|,w/\left|w\right|\right)\leq b_{M}\delta\left(z\right).\label{eq:d-6}
\end{gather}
Here $d_{\mathbb{T}}$ denotes arclength distance on $\mathbb{T}$,
with values in $\left[0,\pi\right]$.

For \prettyref{eq:d-5}, use
\[
\left|d_{\mathbb{D}}\left(0,z\right)-d_{\mathbb{D}}\left(0,w\right)\right|\leq d_{\mathbb{D}}\left(z,w\right)<M,
\]
equivalently, 
\[
e^{-M}<\frac{e^{-d_{\mathbb{D}}\left(0,z\right)}}{e^{-d_{\mathbb{D}}\left(0,w\right)}}<e^{M}
\]
which is 
\begin{equation}
e^{-M}<\frac{\delta\left(z\right)\left(1+\left|w\right|\right)^{2}}{\delta\left(w\right)\left(1+\left|z\right|\right)^{2}}<e^{M}\label{eq:d-7}
\end{equation}
where $d_{\mathbb{D}}\left(0,z\right)=\log\frac{1+\left|z\right|}{1-\left|z\right|}$,
$d_{\mathbb{D}}\left(0,w\right)=\log\frac{1+\left|w\right|}{1-\left|w\right|}$. 

Since $z,w\in\Gamma_{E}$, we have 
\[
3/2<1+\left|z\right|,1+\left|w\right|<2.
\]
Hence the ratio 
\[
\frac{\left(1+\left|w\right|\right)^{2}}{\left(1+\left|z\right|\right)^{2}}
\]
in \prettyref{eq:d-7} has absolute bounds above and below, and \prettyref{eq:d-5}
follows. 

For \prettyref{eq:d-6}, let 
\[
\rho=\left|\frac{z-w}{1-\overline{w}z}\right|.
\]
Then 
\begin{equation}
\rho=\tanh\frac{d_{\mathbb{D}}\left(z,w\right)}{2}\leq\tanh\left(M/2\right)<1.\label{eq:d-8}
\end{equation}
We also have
\[
\left|1-z\overline{w}\right|^{2}=\frac{\delta\left(z\right)\delta\left(w\right)}{1-\rho^{2}}\leq\frac{b_{M}}{1-\tanh^{2}\left(M/2\right)}\delta\left(z\right)^{2}\leq C_{M}\delta\left(z\right)^{2}
\]
where the estimate used \prettyref{eq:d-5} and \prettyref{eq:d-8}.
Thus 
\begin{gather*}
\left|1-z\overline{w}\right|\leq C_{M}\delta\left(z\right)\\
\Downarrow\\
\left|z-w\right|=\rho\left|1-z\overline{w}\right|\leq C_{M}\delta\left(z\right)
\end{gather*}
Because $\left|z\right|,\left|w\right|>1/2$, 
\[
\left|z/\left|z\right|-w/\left|w\right|\right|\leq4\left|z-w\right|\leq C_{M}\delta\left(z\right)
\]
and standard arclength estimate gives \prettyref{eq:d-6}.

\textbf{2.} The $M$-separation shows that $\Lambda$ has no accumulation
point in $\mathbb{D}$. We now show that the set of boundary accumulation
points of $\Lambda$ is exactly $E$:

Let $\lambda_{j}\in\Lambda$ and $\lambda_{j}\to\zeta\in\mathbb{T}$.
Because $\Lambda\subset\Gamma_{E}$, by definition of $\Gamma_{E}$,
\[
\frac{\lambda_{j}}{\left|\lambda_{j}\right|}\in E.
\]
Then $\left|\lambda_{j}\right|\rightarrow1$, $\frac{\lambda_{j}}{\left|\lambda_{j}\right|}\rightarrow\zeta$.
Since $E$ is closed, $\zeta\in E$. 

Conversely, fix $\zeta\in E$ and consider $r\zeta$ with $r\uparrow1$.
Because $\zeta\in E$ and $1/2<r<1$ once $r$ is sufficiently close
to $1$, we have $r\zeta\in\Gamma_{E}$. By maximality, every point
of $\Gamma_{E}$ is within hyperbolic distance $<M$ of some point
in $\Lambda$, so we choose $\lambda_{r}\in\Lambda$ such that 
\[
d_{\mathbb{D}}\left(r\zeta,\lambda_{r}\right)<M.
\]

Now apply \prettyref{eq:d-5} with $z=r\zeta$, $w=\lambda_{r}$.
Since 
\[
\delta\left(r\zeta\right)=1-r^{2},\quad\delta\left(\lambda_{r}\right)=1-\left|\lambda_{r}\right|^{2},
\]
we get 
\[
a_{M}\left(1-r^{2}\right)\leq1-\left|\lambda_{r}\right|^{2}\leq b_{M}\left(1-r^{2}\right).
\]
Therefore, $\left|\lambda_{r}\right|\rightarrow1$ as $r\rightarrow1$.
But, using \prettyref{eq:d-6}, 
\[
d_{\mathbb{T}}\left(\zeta,\frac{\lambda_{r}}{\left|\lambda_{r}\right|}\right)\leq b_{M}\left(1-r^{2}\right)\longrightarrow0.
\]
Combining these, 
\[
\lambda_{r}\longrightarrow\zeta.
\]
Hence every point of $E$ is a boundary accumulation point.

\textbf{3.} Let's verify the upper estimate in \prettyref{eq:4-4},
\[
F_{\Lambda}\left(t\right)\leq C_{M}t^{2}.
\]
For $0<s\leq1$, let 
\[
\Lambda\left(s\right)=\left\{ \lambda\in\Lambda:\frac{s}{2}<\delta\left(\lambda\right)\leq s\right\} .
\]
We need an estimate of $\#\Lambda\left(s\right)$. 

First, by separation, the disks 
\[
D_{\mathbb{D}}\left(\lambda,M/3\right),\quad\lambda\in\Lambda\left(s\right)
\]
are pairwise disjoint. 

Take $\lambda\in\Lambda\left(s\right)$, and consider $z\in D_{\mathbb{D}}\left(\lambda,M/3\right)$.
Notice the argument for \prettyref{eq:d-5} applies more generally
to any two points of $\mathbb{D}$ at bounded hyperbolic distance,
so 
\[
c_{M}\delta\left(\lambda\right)\leq\delta\left(z\right)\leq C_{M}\delta\left(\lambda\right).
\]
But $\lambda\in\Lambda\left(s\right)$ means 
\[
\frac{s}{2}<\delta\left(\lambda\right)\leq s.
\]
Thus
\[
c_{M}s\leq\delta\left(z\right)\leq C_{M}s.
\]
Since $\lambda$ and $z$ were arbitrary, the union of these disks
is contained in a strip, 
\[
\bigcup_{\lambda\in\Lambda\left(s\right)}D_{\mathbb{D}}\left(\lambda,M/3\right)\subset A_{s}:=\left\{ z\in\mathbb{D}:c_{M}s<\delta\left(z\right)<C_{M}s\right\} .
\]

Now use the hyperbolic area, 
\[
Area_{hyp}\left(\cup_{\lambda\in\Lambda\left(s\right)}\cdots\right)=\#\Lambda\left(s\right)A_{M}\leq Area_{hyp}\left(A_{s}\right)\leq\frac{C_{M}}{s}
\]
where $A_{M}=\text{area of }D_{\mathbb{D}}\left(0,M/3\right)>0$.
Thus 
\[
\#\Lambda\left(s\right)\leq\frac{C_{M}}{s}.
\]

Take the partition $\Lambda\left(t\right)$, $\Lambda\left(t/2\right)$,
... for $\left\{ \lambda\in\Lambda:\delta\left(\lambda\right)\leq t\right\} $.
Then 
\begin{align*}
F_{\Lambda}\left(t\right) & =\sum_{\lambda\in\Lambda,\:\delta\left(\lambda\right)\leq t}\delta\left(\lambda\right)^{3}\\
 & =\sum^{\infty}_{k=0}\sum_{\lambda\in\Lambda\left(2^{-k}t\right)}\delta\left(\lambda\right)^{3}\\
 & \leq\sum^{\infty}_{k=0}\#\Lambda\left(2^{-k}t\right)\left(2^{-k}t\right)^{3}\\
 & \leq C_{M}t^{2}\sum^{\infty}_{k=0}2^{-2k}\leq C_{M}t^{2}.
\end{align*}

\textbf{4.} The lower estimate 
\[
F_{\Lambda}\left(t\right)\geq c_{M}m_{\mathbb{T}}\left(E\right)t^{2}
\]
is as follows. 

For $\zeta\in E$ and $0<t\leq1$, let 
\[
z_{\zeta}=\sqrt{1-\alpha t}\zeta,\quad0<\alpha<3/4
\]
Then $z_{\zeta}\in\Gamma_{E}$, by definition. 

By maximality there is $\lambda_{\zeta}\in\Lambda$ with 
\[
d_{\mathbb{D}}\left(z_{\zeta},\lambda_{\zeta}\right)<M.
\]
Applying \prettyref{eq:d-5}, 
\begin{gather*}
a_{M}\delta\left(z_{\zeta}\right)\leq\delta\left(\lambda_{\zeta}\right)\leq b_{M}\delta\left(z_{\zeta}\right)\\
\Updownarrow\\
a_{M}\alpha t\leq\delta\left(\lambda_{\zeta}\right)\leq b_{M}\alpha t.
\end{gather*}
We may choose $\alpha$ so that $b_{M}\alpha<1$, so 
\begin{equation}
c_{M}t\leq\delta\left(\lambda_{\zeta}\right)\leq t.\label{eq:d-9}
\end{equation}

Then \prettyref{eq:d-6} gives
\[
d_{\mathbb{T}}\left(z_{\zeta}/\left|z_{\zeta}\right|,\lambda_{\zeta}/\left|\lambda_{\zeta}\right|\right)\leq b_{M}\delta\left(z_{\zeta}\right).
\]
This gives 
\[
d_{\mathbb{T}}\left(\zeta,\lambda_{\zeta}/\left|\lambda_{\zeta}\right|\right)\leq C_{M}t.
\]

We conclude that the arcs of radius $C_{M}t$ centered at $\lambda/\left|\lambda\right|$
satisfying \prettyref{eq:d-9} cover $E$. Since each such arc has
$m_{\mathbb{T}}$-measure at most $C_{M}t$, there are at least 
\[
\frac{c_{M}m_{\mathbb{T}}\left(E\right)}{t}
\]
such points $\lambda$. Each of these points satisfies 
\[
\delta\left(\lambda\right)^{3}\geq c_{M}t^{3}.
\]
Therefore 
\[
F_{\Lambda}\left(t\right)\geq\frac{c_{M}m_{\mathbb{T}}\left(E\right)}{t}c_{M}t^{3}\geq c_{M}m_{\mathbb{T}}\left(E\right)t^{2}.
\]
\end{proof}

\begin{thm}
\label{thm:4-4} Let $E\subset\mathbb{T}$ be closed and suppose $m_{\mathbb{T}}\left(E\right)>0$.
There exist absolute constants $a,C>0$, and a finite positive atomic
measure 
\[
\mu=\sum\nolimits_{j}w_{j}\delta_{\lambda_{j}},\quad\lambda_{j}\in\mathbb{D},
\]
with $m_{n}=\int\left|z\right|^{2n}d\mu\left(z\right)$, such that
\[
\left\{ z^{n}/\sqrt{m_{n}}\right\} 
\]
is a frame for $L^{2}\left(\mu\right)$ with frame bounds $A_{E},B_{E}$
satisfying 
\begin{equation}
a\leq A_{E}\leq B_{E}\leq\frac{C}{m_{\mathbb{T}}\left(E\right)}.\label{eq:4-6}
\end{equation}

The set of boundary accumulation points of $\left\{ \lambda_{j}\right\} $
is $E$. The points $\lambda_{j}$ may be chosen distinct and bounded
away from zero.

\end{thm}

\begin{proof}
Fix the $M$ from \prettyref{lem:4-2}, and take $\Lambda$ from \prettyref{prop:4-3}.
Define 
\[
\mu=\sum_{\lambda\in\Lambda}\delta\left(\lambda\right)^{3}\delta_{\lambda}.
\]
By \prettyref{eq:4-4}, 
\[
\mu\left(\overline{\mathbb{D}}\right)=F_{\Lambda}\left(1\right)\leq C,
\]
so $\mu$ is finite. The boundary accumulation set of $\Lambda$ is
$E$, hence the outer radius of $\mu$ is one.

Its moments are 
\begin{align}
m_{n} & =\int\left|z\right|^{2n}d\mu\left(z\right)\nonumber \\
 & =\sum_{\lambda\in\Lambda}\delta\left(\lambda\right)^{3}\left|\lambda\right|^{2n},\quad\left|\lambda\right|^{2}=1-\delta\left(\lambda\right)\nonumber \\
 & =\sum_{\lambda\in\Lambda}\delta\left(\lambda\right)^{3}\left(1-\delta\left(\lambda\right)\right)^{n}\label{eq:4-11}\\
 & =\int^{1}_{0}\left(1-t\right)^{n}dF_{\Lambda}\left(t\right).\label{eq:d-12}
\end{align}

\textbf{1.} We need an estimate for $m_{n}$. For $n\geq1$, integration
by parts in \prettyref{eq:d-12} and the upper estimate in \prettyref{eq:4-4}
give 
\begin{align*}
m_{n} & =n\int^{1}_{0}F_{\Lambda}\left(t\right)\left(1-t\right)^{n-1}dt\\
 & \leq Cn\int^{1}_{0}t^{2}\left(1-t\right)^{n-1}dt=\frac{2C}{\left(n+1\right)\left(n+2\right)}.
\end{align*}
The same upper bound holds for $n=0$, since $m_{0}\leq C$; so 
\[
m_{n}\leq\frac{2C}{\left(n+1\right)\left(n+2\right)},\quad n=0,1,2,\ldots
\]

For $n\geq1$, let $t=1/\left(n+1\right)$. Since $\delta\left(\lambda\right)\leq t$,
\[
\left(1-\delta\left(\lambda\right)\right)^{n}\geq\left(1-\frac{1}{n+1}\right)^{n}=\left(1+\frac{1}{n}\right)^{-n}\geq e^{-1}.
\]
Using \prettyref{eq:4-11} and \prettyref{eq:4-4}, we get 
\begin{align*}
m_{n} & \geq e^{-1}\sum_{\delta\left(\lambda\right)\leq1/\left(n+1\right)}\delta\left(\lambda\right)^{3}=e^{-1}F_{\Lambda}\left(\frac{1}{n+1}\right)\geq\frac{cm_{\mathbb{T}}\left(E\right)}{\left(n+1\right)^{2}}.
\end{align*}
For $n=0$, 
\[
m_{0}=F_{\Lambda}\left(1\right)\geq cm_{\mathbb{T}}\left(E\right),
\]
same lower bound holds. 

Therefore, 
\begin{equation}
\frac{c\,m_{\mathbb{T}}\left(E\right)}{\left(n+1\right)^{2}}\leq m_{n}\leq\frac{C}{\left(n+1\right)^{2}},\quad n\geq0.\label{eq:4-5}
\end{equation}

\textbf{2.} Let $\mathcal{H}_{0}$ be the coefficient space associated
with $\mu$ as in \prettyref{eq:b-5}--\prettyref{eq:b-6}, and let
$K$ be its reproducing kernel from \prettyref{eq:b-7}. Thus
\[
\mathcal{H}_{0}=\left\{ \sum\nolimits_{n\geq0}c_{n}z^{n}:\sum\nolimits_{n\geq0}\left|c_{n}\right|^{2}m_{n}<\infty\right\} 
\]
and
\[
K\left(z,w\right)=\sum\nolimits_{n\geq0}u_{n}\left(z\right)\overline{u_{n}\left(w\right)}=\sum\nolimits_{n\geq0}\frac{\left(z\overline{w}\right)^{n}}{m_{n}}.
\]

Recall also the fixed RKHS $\mathcal{B}$ from the beginning of this
section, with kernel 
\[
B\left(z,w\right)=\frac{1}{\left(1-z\overline{w}\right)^{3}}=\sum_{n\geq0}b_{n}z^{n}\overline{w}^{n},\qquad b_{n}=\frac{\left(n+1\right)\left(n+2\right)}{2}.
\]
See \prettyref{eq:4-1}. 

Now we compare $\mathcal{H}_{0}$ with $\mathcal{B}$ through the
coefficients. Since 
\[
\frac{\left(n+1\right)^{2}}{2}\leq b_{n}\leq\left(n+1\right)^{2},
\]
\prettyref{eq:4-5} gives absolute constants $c_{1},C_{1}>0$ such
that 
\[
c_{1}b_{n}\leq\frac{1}{m_{n}}\leq\frac{C_{1}}{m_{\mathbb{T}}\left(E\right)}b_{n},\quad n\geq0.
\]

\textbf{3.} We now return to the frame criterion in \prettyref{thm:2-5}.
For the atomic measure 
\[
\mu=\sum_{\lambda\in\Lambda}\delta\left(\lambda\right)^{3}\delta_{\lambda},
\]
the weights in the notation of \prettyref{sec:2} are 
\[
w_{\lambda}=\delta\left(\lambda\right)^{3}.
\]
The vectors associated with these atoms in \prettyref{eq:2-14} are
therefore 
\[
h_{\lambda}=\sqrt{w_{\lambda}}K_{\lambda}=\delta\left(\lambda\right)^{3/2}K_{\lambda},\quad\lambda\in\Lambda.
\]
By \prettyref{thm:2-5}, the normalized monomials 
\[
\left\{ z^{n}/\sqrt{m_{n}}\right\} _{n\geq0}
\]
form a frame for $L^{2}\left(\mu\right)$ if and only if $\left(h_{\lambda}\right)_{\lambda\in\Lambda}$
is a Riesz sequence in $\mathcal{H}_{0}$. The frame bounds are the
corresponding Riesz bounds. It is therefore enough to estimate the
Riesz bounds of $\left(h_{\lambda}\right)_{\lambda\in\Lambda}$.

We do this by comparison with the fixed space $\mathcal{B}$. Recall
from \prettyref{eq:4-2} that its normalized reproducing kernels are
\[
e_{\lambda}=\frac{B_{\lambda}}{\left\Vert B_{\lambda}\right\Vert _{\mathcal{B}}}=\delta\left(\lambda\right)^{3/2}B_{\lambda}.
\]
The set $\Lambda$ is $M$-separated by \prettyref{prop:4-3}, where
$M$ was chosen from \prettyref{lem:4-2}. Hence $\left(e_{\lambda}\right)_{\lambda\in\Lambda}$
is a Riesz sequence in $\mathcal{B}$ with absolute Riesz bounds.
We compare $\left(h_{\lambda}\right)$ with this sequence using the
coefficient comparison obtained in step \textbf{2.}

For every finitely supported family $\left(c_{\lambda}\right)$, 
\begin{align*}
\left\Vert \sum\nolimits_{\lambda}c_{\lambda}h_{\lambda}\right\Vert ^{2}_{\mathcal{H}_{0}} & =\sum_{n\geq0}\frac{1}{m_{n}}\left|\sum\nolimits_{\lambda}c_{\lambda}\delta\left(\lambda\right)^{3/2}\overline{\lambda}^{n}\right|^{2},\\
\left\Vert \sum\nolimits_{\lambda}c_{\lambda}e_{\lambda}\right\Vert ^{2}_{\mathcal{B}} & =\sum_{n\geq0}b_{n}\left|\sum\nolimits_{\lambda}c_{\lambda}\delta\left(\lambda\right)^{3/2}\overline{\lambda}^{n}\right|^{2}.
\end{align*}
Therefore 
\[
c_{1}\left\Vert \sum\nolimits_{\lambda}c_{\lambda}e_{\lambda}\right\Vert ^{2}_{\mathcal{B}}\leq\left\Vert \sum\nolimits_{\lambda}c_{\lambda}h_{\lambda}\right\Vert ^{2}_{\mathcal{H}_{0}}\leq\frac{C_{1}}{m_{\mathbb{T}}\left(E\right)}\left\Vert \sum\nolimits_{\lambda}c_{\lambda}e_{\lambda}\right\Vert ^{2}_{\mathcal{B}}.
\]
By \prettyref{lem:4-2}, $\left(e_{\lambda}\right)_{\lambda\in\Lambda}$
is a Riesz sequence in $\mathcal{B}$ with absolute bounds. Hence
$\left(h_{\lambda}\right)_{\lambda\in\Lambda}$ is a Riesz sequence
in $\mathcal{H}_{0}$ with an absolute lower bound and upper bound
at most $C/m_{\mathbb{T}}\left(E\right)$. By \prettyref{thm:2-5},
the normalized monomials form a frame for $L^{2}\left(\mu\right)$
with bounds satisfying \prettyref{eq:4-6}.

The remaining assertions follow from \prettyref{prop:4-3}. The atoms
are distinct, bounded away from zero, have no accumulation point in
$\mathbb{D}$, and their boundary accumulation set is $E$.
\end{proof}

We can now return to the original operator problem. 
\begin{thm}
\label{thm:4-5} Let $E\subset\mathbb{T}$ be closed. The following
are equivalent. 
\begin{enumerate}
\item $E$ has positive Haar measure, $m_{\mathbb{T}}\left(E\right)>0$. 
\item There are an invertible diagonal normal operator $T$ with simple
eigenvalues and a cyclic vector $g$ such that $\sigma_{\mathrm{e}}\left(T\right)=E$
and 
\[
\left(\frac{T^{n}g}{\left\Vert T^{n}g\right\Vert }\right)_{n\geq0}
\]
is a frame. 
\end{enumerate}
Moreover, there are absolute constants $a,C>0$ such that the frame
in the second assertion may be chosen with bounds $A_{E},B_{E}$ satisfying
\begin{equation}
a\leq A_{E}\leq B_{E}\leq\frac{C}{m_{\mathbb{T}}\left(E\right)}.\label{eq:4-7}
\end{equation}

Conversely, if a normal operator has a normalized orbit frame with
upper frame bound $B$ and $\sigma_{\mathrm{e}}\left(T\right)=E\subset\mathbb{T}$,
then 
\begin{equation}
Bm_{\mathbb{T}}\left(E\right)\geq1.\label{eq:4-8}
\end{equation}
\end{thm}

\begin{proof}
Suppose first that $m_{\mathbb{T}}\left(E\right)>0$. Take the measure
$\mu$ from \prettyref{thm:4-4}, let $T=M_{z}$ on $L^{2}\left(\mu\right)$,
and let $g=1$. Then the normalized orbit of $g$ is the normalized
monomial frame from \prettyref{thm:4-4}, with bounds satisfying \prettyref{eq:4-7}.
Since the frame is complete, $g$ is cyclic.

The atoms of $\mu$ are distinct and bounded away from zero, so $T$
is invertible and has simple eigenvalues. They have no accumulation
point in $\mathbb{D}$, and their boundary accumulation set is $E$.
By the operator interpretation in \prettyref{sec:3}, $\sigma_{\mathrm{e}}\left(T\right)=E$. 

Conversely, suppose that $T$ is normal, $\sigma_{\mathrm{e}}\left(T\right)=E\subset\mathbb{T}$,
and the normalized orbit of $g$ is a frame with upper frame bound
$B$. Since the frame is complete, $g$ is cyclic. Normalize the outer
radius of the scalar spectral measure as in \prettyref{sec:3}. By
\prettyref{thm:2-3}, after this normalization the essential spectrum
is contained in $\mathbb{T}$. Since scaling $T$ by the reciprocal
of the outer radius scales its essential spectrum by the same factor,
and $E$ is a nonempty subset of $\mathbb{T}$, the outer radius must
be one.

Thus the outer spectral set $E_{\mu_{g}}$ from \prettyref{sec:3}
is the essential spectrum $E$. Applying \prettyref{cor:3-6} gives
\prettyref{eq:4-8}.
\end{proof}

\bibliographystyle{amsalpha}
\bibliography{orbit}

\end{document}